\documentclass[graybox]{svmult}

\usepackage{mathptmx}       
\usepackage{helvet}         
\usepackage{courier}        
\usepackage{type1cm}        
\usepackage{makeidx}         
\usepackage{graphicx}        
\usepackage{multicol}        
\usepackage[bottom]{footmisc}
\usepackage{amsmath, amsfonts, stmaryrd, amssymb, dsfont}
\usepackage[normalem]{ulem}

\newcommand{\N}{\mathbb{N}}                                              
\newcommand{\Z}{\mathbb{Z}}                                                                                         
\newcommand{\Q}{\mathcal{Q}}                                              
\newcommand{\R}{\mathbb{R}}                                              

\newcommand{\T}{\mathbb{T}}

\newcommand{\p}{\mathbf{p}}

\newcommand{\cL}{\mathcal{L}}
\newcommand{\cH}{\mathcal{H}}
\newcommand{\cQ}{\mathcal{Q}}

\newcommand{\cX}{\mathcal{X}}

\DeclareMathOperator{\Ent}{Ent}

\newtheorem{thm}{\textsc{Theorem}}[section]

\newtheorem{lem}[thm]{\textsc{Lemma}}
\newtheorem{prop}[thm]{\textsc{Proposition}}
\newtheorem{defn}{\textsc{Definition}}[section]

\newtheorem{rem}{\textsc{Remark}}[section]
\newtheorem{exple}{\textsc{Example}}[section]

\makeindex             

\begin{document}

\title*{The gradient flow approach to hydrodynamic limits for the simple exclusion process}

\author{Max Fathi and Marielle Simon}
\institute{Marielle Simon \at Departamento de Matem\'atica, PUC, Rua Marqu\^es de S\~ao Vicente 225, 22453-900, Rio de Janeiro \at\email{marielle.simon@mat.puc-rio.br} \newline Max Fathi \at Department of Mathematics, University of California, Berkeley \at\email{mfathi@phare.normalesup.org}}
%
%
\maketitle

\abstract{We present a new approach to prove the macroscopic hydrodynamic behaviour for interacting particle systems, and as an example we treat the well-known case of the symmetric simple exclusion process (SSEP). More precisely, we characterize any possible limit of its empirical density measures as solutions to the heat equation by passing to the limit in the gradient flow structure of the particle system.
 }

\section{Introduction}
\label{sec:1}
The aim of this work is to show how one can use gradient flow structures to prove convergence to the hydrodynamic limit for interacting particle systems. The exposition is focused on the case of the symmetric simple exclusion process on the discrete $d$-dimensional torus, but the strategy can be adapted to other reversible particle systems, such as zero-range processes (see \cite{MR1707314} for the definitions of these models).  Gradient flows are ordinary differential equations of the form
$$\dot{x}(t) = -\nabla V(x(t)),$$ where $\nabla V$ denotes the gradient of the function $V$. 
De Giorgi and his collaborators showed in \cite{DGMT} how to give a meaning to solutions to such equations in the setting of metric spaces: these solutions are called minimizing-movement solutions, or curves of maximal slope. When considering the case of spaces of probability measures, one can use this notion to rewrite the partial differential equations governing the time evolution of the laws of diffusion processes, such as the heat equation, as gradient flows, for which $V$ is the entropy with respect to the optimal transport (or Wasserstein) distance. We refer to \cite{AGS} for more details. This framework can be adapted to the case of reversible Markov chains on finite spaces. This was proven independently by Maas \cite{M} and Mielke \cite{Mie}, who both developed a discrete counterpart to the Lott-Sturm-Villani theory of lower bounds on Ricci curvature for metric spaces. 

These gradient flow structures are a powerful tool to study convergence of sequences of dynamics to some limit. Two main strategies have already been developed. One of them consists in using the discrete (in time) approximation schemes suggested by the gradient flow structure (see for example \cite{ASZ}). The second one, which we shall use here, consists in characterizing gradient flows in terms of a relation between the energy function and its variations, and passing to the limit in this characterization. It was first developed by Sandier and Serfaty in \cite{SS}, and then generalized in \cite{Ser}. This strategy can be combined with the gradient flow structure of \cite{M,Mie} to prove convergence to some scaling limit for interacting particle systems. This was recently done for chemical reaction equations in \cite{MM} and mean-field interacting particle systems on graphs in \cite{EFLS}. 

Gradient flow structures are also related to large deviations, at least when considering diffusion processes, see  \cite{ADPZ2011,F}. While we only present here the case of the SSEP, the technique is fairly general, and can be adapted to other reversible interacting particle systems. For example, the adaptation of the proof to the case of a zero-range process on the lattice (with nice rates) is quite straightforward. It would be very interesting to apply this method to obtain other PDEs, more degenerate  than the heat equation, as hydrodynamic limits of some interacting particle system: for instance, porous medium  and fast diffusion equations (see \cite{otto, vazquez}) also have a gradient flow structure, and are not directly solvable by standard techniques.

The plan of the sequel is as follows: in Section \ref{sec:framework}, we present the gradient flow framework for Markov chains on discrete spaces developed in \cite{M,Mie}. In Section \ref{sec:scal}, we expose the setup for proving convergence of gradient flows. Finally, in Section \ref{sec:ssep}, we investigate the symmetric simple exclusion process, and reprove the convergence to its hydrodynamic limit.

\section{Gradient flow structure for reversible Markov chains}
\label{sec:framework}

\subsection{Framework}

We start by describing the gradient flow framework for Markov chains on discrete spaces. The presentation we use here is the one of \cite{M}. We consider an irreducible continuous time reversible Markov chain on a finite space $\mathcal{X}$ with kernel $K : \mathcal{X}\times\mathcal{X} \to \R_+$ and invariant probability measure $\nu$. Let $\mathcal{P}(\mathcal{X})$ (resp. $\mathcal{P}_+(\mathcal{X})$)
be the set of probability  densities  (resp. positive) with respect to $\nu$. The probability law $\rho_t\nu$ of the Markov chain at time $t$ satisfies the evolution equation 
\begin{equation} 
\dot{\rho}_t(x)+\sum_{y \in \mathcal{X}} (\rho_t(x)-\rho_t(y))K(x,y)=0,\quad \text{ for all } x \in \mathcal{X}.
\label{eq:evolution}
\end{equation}
Hereafter we denote by $\dot{\rho}_t(x)$ the derivative with respect to time of the function $(t,x)\mapsto \rho_t(x)$. Given a function $\psi : \mathcal{X} \to \R$, we  define $\nabla \psi(x,y): = \psi(y) - \psi(x)$.  The discrete divergence of a function $\Phi : \cX \times \cX \longrightarrow \R$ is defined as
$$\operatorname{div}(\Phi)(x) := \frac{1}{2}\sum_{y\in\cX} (\Phi(x,y) - \Phi(y,x))K(x,y).$$
With these definitions, we have the integration by parts formula
\begin{equation}\sum_{x,y\in\cX} \nabla \psi(x,y) \Phi(x,y)K(x,y)\nu(x) = -\sum_{x\in\cX} \psi(x)\operatorname{div}(\Phi)(x)\nu(x).\label{eq:ipp}\end{equation}
Let us introduce three notions we shall use to define the gradient flow structure: 
\begin{defn}
\begin{enumerate}
\item The \emph{relative entropy with respect to} $\nu$ is defined as
\begin{equation*}
{\Ent}_\nu(\rho) := \sum_{x \in \mathcal{X}} \nu(x)\rho(x)\log \rho(x), \quad \text{ for } \rho \in \mathcal{P}(\mathcal{X}),
\end{equation*} with the convention that $\rho(x)\log \rho(x)=0$ if $\rho(x)=0$. We  sometimes denote $\mathcal{H}(\rho):={\Ent}_\nu(\rho)$, whenever  $\nu$ is fixed and no confusion arises.

\item The \emph{symmetric Dirichlet form} is given for two real-valued functions $\phi,\psi$  by
\begin{equation*}
\mathcal{E}(\phi, \psi) := \frac{1}{2}\sum_{x,y \in \mathcal{X}} (\phi(y) - \phi(x))(\psi(y) - \psi(x))K(x,y)\nu(x),
\end{equation*} 

\item The \emph{Fisher information} (or entropy production) writes as 
$\mathcal{I}(\rho) := \mathcal{E}(\rho, \log \rho )$.
\end{enumerate}
\end{defn}

Notice that $\Ent_{\nu}$ is the \emph{mathematical} entropy, and not the physical entropy. It decreases along solutions of \eqref{eq:evolution}, so in physical terms it plays the role of a free energy. We call $\mathcal{I}$ the entropy production since along  solutions of \eqref{eq:evolution} we have \[\frac{d}{dt}\Ent_{\nu}(\rho_t) = -\mathcal{I}(\rho_t).\] 

\subsection{Continuity equation}
We introduce the \emph{logarithmic mean} $\Lambda(a,b)$ of two non-negative numbers $a,b$ as 
\begin{equation*}
\Lambda(a,b)= \int_0^1 a^s b^{1-s} ds =   \frac{b-a}{\log(b)-\log(a)} \quad \text{ if } a \neq b, a>0, b>0,
\end{equation*}
and $\Lambda(a,a)=a$, and also $\Lambda(a,b)=0$ if $a=0$ or $b=0$. 
The mean $\Lambda$ satisfies: 
\begin{equation}
\sqrt{ab} \le \Lambda(a,b) \le (a+b)/{2}, \qquad \text{ for all } a \ge 0, b \ge 0. \label{eq:logarith_ineq}
\end{equation}
Let us now define, for $\rho \in \mathcal{P}_+(\mathcal{X})$, its \emph{logarithmic mean} $\hat\rho$ defined on $\mathcal{X}\times\mathcal{X}$ as
\begin{equation}
\hat{\rho}(x,y) := \Lambda(\rho(x),\rho(y)).
\end{equation}
In order to define a suitable metric on $\mathcal{P}(\mathcal{X})$, we need a representation of curves as solving a continuity equation: 
\begin{lem}
Given a smooth flow of  positive probability densities $\{\rho_t\}_{t \geq 0}$ on $\cX$, there exists a function $(t,x)\mapsto \psi_t(x)$ such that the following continuity equation holds for any $t \ge 0$ and $x \in \cX$:
\begin{equation} \label{continuity_equation}
\dot\rho_t(x) + \sum_{y \in \mathcal{X}}  (\psi_t(y) - \psi_t(x))K(x,y)\hat{\rho}_t(x,y) = 0.
\end{equation}
Moreover, for any $t\ge 0$, $\psi_t(\cdot)$ is unique up to an additive constant.
\end{lem}

We refer to \cite[Section 3]{M} for the proof.


\begin{defn} Given $\rho \in \mathcal{P}(\mathcal{X})$ and $\psi:\cX\to\R$, we define \emph{the action}
\[\mathcal{A}(\rho, \psi) := \frac12 \sum_{x,y \in \mathcal{X}}  (\psi(y) - \psi(x))^2\; \hat{\rho}(x,y)K(x,y)\nu(x) \geq 0.\]
\end{defn}
A distance between two probability densities $(\rho_0,\rho_1)$ could then be defined as the infimum of the action of all curves $\{\rho_t,\psi_t\}_{t\in[0,1]}$ linking these densities, as  was done in \cite{M}. However, we do not need to introduce that metric here, since we shall only use the formulation of gradient flows as minimizing-movement curves, as follows: 

\begin{prop} \label{prop_formulation_grad_flow}
Let $\{\rho_t\}_{t \ge 0}$ be a smooth flow of probability densities on $\mathcal{X}$, and let $\{\psi_t\}_{t\ge 0}$ be such that the continuity equation \eqref{continuity_equation} holds. Then, for any $T>0$, 
\begin{equation}
\Ent_\nu(\rho_T) - \Ent_\nu(\rho_0) + \frac{1}{2}\int_0^T{\mathcal{I}(\rho_t)dt} + \frac{1}{2}\int_0^T{\mathcal{A}(\rho_t, \psi_t)dt} \geq 0,\label{eq:gradient}
\end{equation}
with equality if and only if $\{\rho_t\}_{t \ge 0}$ is the flow of the Markov process on $\mathcal{X}$ with kernel $K$ and invariant measure $\nu$, solution to \eqref{eq:evolution}.
\end{prop}

This is the analogue of the characterization of solutions to $\dot{x_t} = -\nabla V(x_t)$ on $\R^d$ as the only curves for which the non-negative functional \[V(x_T) - V(x_0) + \frac{1}{2}\int_0^T{|\nabla V(x_t)|^2dt} + \frac{1}{2}\int_0^T{|\dot{x}_t|^2dt}\] cancels. Hence in the framework of Markov chains, $\Ent_{\nu}$ plays the role of $V$, and the entropy production $\mathcal{I}$ plays the role of $|\nabla V|^2$. 

\begin{proof}[of Proposition \ref{prop_formulation_grad_flow}]
Denote $\mathcal{H}(\rho)=\Ent_\nu(\rho)$. We have
\begin{align*}
\mathcal{H}(\rho_T) - \mathcal{H}(\rho_0) & = \int_0^T{\frac{d}{dt}\big(\mathcal{H}(\rho_t)\big)dt} = \int_0^T{\sum_{x \in \mathcal{X}} \nu(x)\frac{d}{dt}\big(\rho_t(x) \log \rho_t(x)\big) dt} \\
&= \int_0^T{\sum_{x \in \mathcal{X}} \nu(x)\dot{\rho}_t(x) \log \rho_t(x) dt}.
\end{align*}
Using the reversibility  of the invariant measure $\nu$, we write
\begin{align*}
\mathcal{H}(\rho_T) &- \mathcal{H}(\rho_0) = -\int_0^T\sum_{x,y \in \mathcal{X}} \nu(x)(\psi_t(y) - \psi_t(x))K(x,y)\hat{\rho}_t(x,y) \log \rho_t(x) \; dt  \\
&= \int_0^T{\frac{1}{2}\sum_{x,y \in \mathcal{X}} (\psi_t(y) - \psi_t(x))(\log \rho_t(y) - \log \rho_t(x))\hat{\rho}_t(x,y)K(x,y)\nu(x)\; dt} \\
&\geq -\frac{1}{4}\int_0^T{\sum_{x,y \in \mathcal{X}} (\psi_t(x) - \psi_t(y))^2\hat{\rho}_t(x,y)K(x,y)\nu(x)\; dt} \\
& \hspace{1cm} - \frac{1}{4}\int_0^T{\sum_{x,y \in \mathcal{X}} (\log \rho_t(x) - \log \rho_t(y))^2\hat{\rho}_t(x,y)K(x,y)\nu(x)\; dt}  \\
&= -\frac{1}{2}\int_0^T{\mathcal{A}(\rho_t, \psi_t)\; dt} -\frac{1}{2}\int_0^T{\mathcal{I}(\rho_t)\; dt}
\end{align*}
with equality if and only if, for all $x,y \in \mathcal{X}$ and almost every $t \in [0, T]$, we have 
$$\psi_t(x) - \psi_t(y) = \log \rho_t(y) - \log \rho_t(x)$$
which is equivalent to saying that for almost every $t$ and for every $x$ we have
$$\dot\rho_t(x) + \sum_{y \in \mathcal{X}} (\log \rho_t(x) - \log \rho_t(y))\hat{\rho}_t(x,y)K(x,y) = 0.$$
\end{proof}

\section{Scaling limits and gradient flows}
\label{sec:scal}

With the formulation of Proposition \ref{prop_formulation_grad_flow}, we can use the approach of Sandier and Serfaty  \cite{Ser} to study convergence of sequences of Markov chains to a scaling limit. Let $(K_n)$ be a sequence of reversible Markov kernels on finite spaces $\mathcal{X}_n$, and let $(\nu_n)$ be the sequence of invariant measures on $\mathcal{X}_n$. Since we wish to investigate the asymptotic behaviour of the sequence of random processes, it is much more convenient to work in a single space $\mathcal{X}$ that contains all the $\mathcal{X}_n$. Hence we shall assume that we are given a space $\mathcal{X}$ and a collection of embeddings $\p_n : \mathcal{X}_n \longrightarrow \mathcal{X}$. In practice, the choice of  $\mathcal{X}$ and $\p_n$ is suggested by the model under investigation. In the next section, which is focused on the simple exclusion process on the torus, the embeddings will map a configuration $\eta$ onto the associated empirical measure $\pi^n(\eta)$ (see \eqref{eq:pi}, Subsection \ref{ssec:model}). Such embeddings immediately define embeddings of $\mathcal{P}(\mathcal{X}_n)$ into $\mathcal{P}(\mathcal{X})$. 

In order to simplify the exposition below, we adopt the following convention: whenever we say that a sequence $(x_n)$ of elements of  $\mathcal{X}_n$ converges to $x \in \mathcal{X}$, we shall mean that $\p_n(x_n) \longrightarrow x$ as $n$ goes to infinity. In particular, the topology used for convergence is implicitly the topology of $\mathcal{X}$, which is assumed to be a separable complete metric space. 
%
The strategy is to characterize possible candidates for the limit  as gradient flows. For that purpose we give a definition of minimizing-movement curves in the metric setting:


\begin{defn}
 Let $(\cX,d)$ be a complete metric space. The \emph{gradient flows} of an energy functional $\mathcal{H} : \cX \to \R$ with respect to the metric $d$  are the curves $\{m_t\}$ s.t. 
$$\mathcal{H}(m_T) - \mathcal{H}(m_0) + \frac{1}{2}\int_0^T{g(m_t)dt} + \frac{1}{2}\int_0^T{|\dot{m}_t|^2dt} = 0,$$ 
where $g$ is the local slope for $\mathcal{H}$, defined as 
\[g(m) := \limsup_{\widetilde m \to m}\frac{\mathcal{H}(m) - \mathcal{H}(\widetilde m)}{d(m,\widetilde m)} \quad \text{ and } \quad |\dot{m}_t| = \limsup_{h \to 0} \frac{d(m_t, m_{t+h})}h.\]
\end{defn}

\begin{rem}
This is not a complete definition. To make it correct, we should introduce the notion of \emph{absolutely continuous curves}, whose slopes are well defined. This is not a real issue here, since we shall only use it for reversible Markov chains (for which the notions have already been well defined previously for curves of strictly positive probability measures) and the heat equation, for which smooth curves of strictly positive functions do not cause any issue (see the next example). We refer to \cite{AGS,Ser} for a more rigorous discussion of the issues in the metric setting.
\end{rem}

\begin{exple}[Heat equation]\label{example}
Let us consider the parabolic PDE
\begin{equation}
\frac{\partial m}{\partial t} = \frac{\partial^2 m}{\partial \theta^2}, \qquad t\geq 0, \; \theta \in (0,1).
\label{eq:heat}
\end{equation} 
We know from \cite{ADPZ2013} that \eqref{eq:heat} is associated to a gradient flow, since we have: 
\begin{multline*}\int_0^1h(m(T,\theta))d\theta - \int_0^1h(m(0,\theta))d\theta \\ + \frac{1}{2}\int_0^T{\int_{0}^1{m(1-m)\Big(\frac{\partial (h'(m))}{\partial \theta}\Big)^2d\theta}dt}+ \frac{1}{2}\int_0^T{\Big\Vert\frac{\partial m}{\partial t}\Big\Vert^2_{-1,m}dt}=0,\end{multline*}
with $h(x) = x\log x + (1-x)\log (1-x)$ and, given $u:[0,1] \to \R$, 
\[||u||_{-1,m}^2 := \sup_J\bigg\{ 2\int_0^1J( \theta)u(\theta)d\theta - \int_0^1m(1-m)\big(J'(\theta)\big)^2d\theta\bigg\},\] 
where the supremum is over all smooth test functions $J$.
\end{exple}

We now state the main result of that section. Hereafter, when we assert that a sequence of curves of probability measures (in $\mathcal{P}(\mathcal{X}_n) \hookrightarrow \mathcal{P}(\cX)$) converges to a deterministic curve $\{m_t\}$ (in $\mathcal{X}$), we mean that it converges to a curve of Dirac measures $\{\delta_{m_t}\}$. Definition \ref{def:conv_loi} below gives a more precise meaning in the case of particle systems. It is important to stress that we study convergence of probability measures, which are deterministic objects.

\begin{thm}\label{theo:grad_flow} 
Let $(a_n)$ be an increasing diverging sequence of positive numbers. We first assume that the topology on $\mathcal{P}(\cX)$ has the following property: 
\begin{quote} \normalsize $(\mathbf{P})$ For any sequence $(\rho_t^n\nu_n)$  of smooth curves of  positive probability measures that converges to some deterministic curve $\{m_t\}$, the following inequalities hold: 
\begin{align}\label{eq:conv0}\liminf_{n \to \infty} & \frac{1}{a_n}\Ent_{\nu_n}(\rho_T^n) \geq \mathcal{H}(m_T)\\
\liminf_{n \to \infty} &\frac{1}{a_n}\int_0^T{\mathcal{I}_n(\rho_t^n)dt} \geq \int_0^T{g(m_t)dt} \label{eq:bound_entropy}\\
\liminf_{n \to \infty}&  \frac{1}{a_n}\int_0^T{\mathcal{A}_n(\rho_t^n, \psi_{t}^n)dt} \geq \int_0^T{|\dot{m}_t|^2dt}, \label{eq:bound_slopes}
\end{align}
where $\psi_t^n$ is such that $(\rho_t^n,\psi_t^n)$ solves \eqref{continuity_equation}. 
\end{quote}

Now consider a sequence $(\rho_t^n\nu_n)$ of gradient flows (so that there is equality in \eqref{eq:gradient}), assume that the initial sequence $(\rho_0^n\nu_n)$ does converge in distribution to some $m_0$, and that moreover
\[\lim_{n \to \infty} \frac{1}{a_n}\Ent_{\nu_n}(\rho_0^n) = \mathcal{H}(m_0).\]
Then, any possible weak limit $\{m_t\}$ of $(\rho_t^n\nu_n)$ is almost surely a gradient flow of the energy $\mathcal{H}$, starting from $m_0$. 
In particular, if gradient flows starting from a given initial data are unique, $(\rho^n_t\nu_n)$ weakly converges to a Dirac measure concentrated on the unique gradient flow of $\cH$ starting from $m_0$.

Moreover, for any $t \in [0,T]$, we have
$$\frac{1}{a_n}\Ent_{\nu_n}(\rho_t^n) \xrightarrow[n\to\infty]{} \mathcal{H}(m_t).$$
\end{thm}

Above $(a_n)$ is a sequence of weights that corresponds to the correct scaling of the system. For particle systems on the discrete torus of length $n$ in dimension $d$ under diffusive scaling, we would take $a_n = n^d$.
This result is a slight variation of the abstract method developed in \cite{Ser}, to which we refer for more details. The main difference (apart from the setting which is restricted to gradient flows in spaces of probability measures arising from reversible Markov chains) is that we consider curves of probability measures that converge to a deterministic curve, rather than any possible limit.

One of the interesting features of this technique is that it does not require an assumption of uniform semi-convexity on the sequence of relative entropies, which can be hard to establish for interacting particle systems (see \cite{EM} for the general theory of geodesic convexity of the entropy for Markov chains, and \cite{FM} for the study of this property for interacting particle systems on the complete graph). Such an assumption of semi-convexity is known as a lower bound on Ricci curvature for the Markov chain, by analogy with the situation for Brownian motion on a Riemannian manifold. For the simple exclusion on the discrete torus, it seems reasonable to conjecture that curvature is non-negative, but this is still an unsolved problem.

\begin{proof}
First of all, for any weak limit  $\mathcal{Q}$ of the laws of the trajectories, we also have 
\begin{equation*}
\liminf_{n \to \infty} \frac{1}{a_n}\int_0^T{\mathcal{E}(\rho_t^n, \log \rho_t^n)dt} \geq \mathcal{Q}\left[\int_0^T{g(m_t)dt} \right]
\end{equation*}
and 
\begin{equation*}
\liminf_{n \to \infty} \frac{1}{a_n}\int_0^T{\mathcal{A}(\rho_t^n, \psi_{t}^n)dt} \geq \mathcal{Q}\left[\int_0^T{|\dot{m}_t|^2dt}\right],
\end{equation*}
where we denote by $\{m_t\}$ a random trajectory with law $\mathcal{Q}$.
This is a direct consequence of the following lemma (whose proof is given below): 
\begin{lem}\label{lem:DetStochLim}
Let $(f_n)$ be a sequence of real-valued, non-negative functions on a space $(\Omega,\mathbb{P})$, and assume that there exists a function $f$ such that for any sequence of random variables $(X_n)$ that converges in law to a deterministic limit $x$, we have
$$\liminf_{n\to\infty} \hspace{1mm} \mathbb{E}[f_n(X_n)] \geq f(x).$$
Then, for any sequence $(X_n)$ of random variables that converges in law to a random variable $X_{\infty}$, we have
$$\liminf_{n\to\infty} \hspace{1mm} \mathbb{E}[f_n(X_n)] \geq \mathbb{E}[f(X_{\infty})].$$
\end{lem}
We now use Proposition \ref{prop_formulation_grad_flow} with the gradient flows $\{\rho_t^n\nu_n\}$, and pass to the limit in 
\[\frac{1}{a_n}\bigg(\Ent_{\nu_n}(\rho^n_T) - \Ent_{\nu_n}(\rho^n_0)+ \frac{1}{2}\int_0^T{\mathcal{E}(\rho_t^n, \log \rho_t^n)dt} + \frac{1}{2}\int_0^T{\mathcal{A}(\rho_t^n, \psi_{t}^n)dt}\bigg) = 0,\] 
and therefore
\begin{equation} \label{eq_conv_proof_ss}
\mathcal{Q}[\mathcal{H}(m_T)] - \mathcal{H}(m_0) + \frac{1}{2}\mathcal{Q}\left[\int_0^T{g(m_t)dt} \right] + \frac{1}{2}\mathcal{Q}\left[\int_0^T{|\dot{m}_t|^2dt}\right] \leq 0.
\end{equation}
Since the above quantity is an expectation of a non-negative functional, we see that
$$\mathcal{H}(m_T) - \mathcal{H}(m_0) + \frac{1}{2}\int_0^T{g(m_t)dt} + \frac{1}{2}\int_0^T{|\dot{m}_t|^2dt} = 0,\quad  \mathcal{Q}\text{-almost surely}. $$
This means that $\mathcal{Q}$-almost surely, $\{m_t\}$ is a gradient flow of $\mathcal{H}$. If uniqueness of gradient flows with initial condition $m_0$ holds, convergence immediately follows.

Convergence of the relative entropy at time $T$  necessarily holds, since otherwise it would contradict \eqref{eq_conv_proof_ss}. Finally, it also holds at any other time $t \in [0,T]$, since one can rewrite the same result on the time-interval $[0,t]$.
\end{proof}

We still have to prove Lemma \ref{lem:DetStochLim}. This proof is taken from \cite{EFLS}.
\begin{proof}[of Lemma \ref{lem:DetStochLim}]
Consider a sequence $(X_n)$ that converges in law to a random variable $X_{\infty}$. Using the almost-sure representation theorem, there exists a sequence $(Y_n)$ such that for any $n$, $Y_n$ has the same law as $X_n$, and  $(Y_n)$ almost surely converges to $Y_{\infty}$. If we condition the whole sequence on the event $\{Y_{\infty} = y\}$, then $(Y_n)$ almost surely converges to $y$. Then we have, using Fatou's lemma
\begin{align*}
\liminf_{n\to\infty} \hspace{1mm} \mathbb{E}[f_n(X_n)]& = \liminf_{n\to\infty} \hspace{1mm} \mathbb{E}[f_n(Y_n)] = \liminf_{n\to\infty} \hspace{1mm} \mathbb{E}_{Y_{\infty}}\big[\mathbb{E}[f_n(Y_n)|Y_{\infty}]\big] \\
&\geq \mathbb{E}_{Y_{\infty}}\left[\liminf_{n\to\infty} \hspace{1mm} \mathbb{E}[f_n(Y_n)|Y_{\infty}]\right] \geq \mathbb{E}_{Y_{\infty}}\left[f(Y_{\infty})\right] = \mathbb{E}[f(X_{\infty})].
\end{align*}
\end{proof}

\section{Symmetric Simple Exclusion Process (SSEP)}
\label{sec:ssep}

\subsection{Model: definitions and notations} \label{ssec:model}
To make notations easier we consider the interacting particle systems on the one-dimensional torus $\T_n=\{0,...,n-1\}$, but the result is valid in any dimension $d\ge 1$. Let us define $\cX_n:=\{0,1\}^{\T_n}$, $\cX:=\{0,1\}^{\Z}$, and $\T= [0,1)$  the continuous torus,  We create a Markov process $\{\eta_t^n \; ; \; t \ge 0 \}$ on the state space $\cX_n$, which satisfies for any $\eta \in \cX_n$: \begin{itemize}
\item $\eta(i)=1$ if there is a particle at site $i \in \T_n$,
\item  $\eta(i)=0$ if the site $i$ is empty,
\item any particle waits independently an exponential time and then jumps to one of its neighbouring sites with probability $1/2$, provided that the chosen site is empty.
\end{itemize} 
We are looking at the evolution of the Markov process in the diffusive time scale, meaning that time is accelerated by $n^2$. The generator is  given for $f:\cX_n\to\R$ by
\[\cL_n(f)(\eta):=  n^2\sum_{i\in\T_n} \eta(i)(1-\eta(i+1)) (f(\eta^{i,i+1})-f(\eta)), \] 
where $\eta^{i,j}$ is the configuration obtained from $\eta$ exchanging the occupation variables $\eta(i)$ and $\eta(j)$. The hydrodynamics behavior of the SSEP is well-known, and we refer the reader to \cite{MR1707314} for a survey.  Let $\nu_\alpha^n$ be the  Bernoulli product measure of parameter $\alpha \in (0,1)$,  the invariant measures for the dynamics. Under $\nu_\alpha^n$, the variables $\{\eta(i)\}_{i \in \T_n}$ are independent with marginals given by 
\[ \nu_\alpha^n\{\eta(i)=1\}=\alpha=1-\nu_\alpha^n\{\eta(i)=0\}.\] Let us fix once and for all $\alpha \in (0,1)$ and denote by $\rho_t^n$ the probability density of the law of $\eta_t^n$ (whose time evolution is generated by $n^2\cL_n$) with respect to $\nu_\alpha^n.$

To prove convergence, we need to embed our particle configurations in a single metric space. For each configuration $\eta \in \cX_n$, we construct a measure on $\cX$ associated to $\eta$, denoted by $\pi^n(\eta)$. We do it here through the \emph{empirical measures}:
\begin{equation}
\pi_t^n(d\theta):=\pi^n(\eta_t^n)(d\theta)=n^{-1} \sum_{i \in \T_n} \eta^n_{t}(i) \delta_{i/n}(d\theta),\label{eq:pi}
\end{equation}
where $\delta_\theta$ stands for the Dirac measure concentrated on $\theta \in \T$. 
Let us denote by $\mathcal{M}_+=\mathcal{M}_+(\cX)$ the space of finite positive measures on $\cX$ endowed with the weak topology. Assume moreover that, for each $n$, $\pi^n:\cX_n\to \mathcal{M}_+$ is a continuous function.  Our goal is to prove the convergence of the flow of measures $(\pi^n(\eta_t^n))$. In particular, $\pi^n$ inherits the Markov property from $\eta^n$. 




We start by defining properly two notions of convergence. For any function $G:\cX \to \R$ and any measure $\pi$ on $\cX$, we denote by $\langle \pi,G\rangle$ the integral of $G$ with respect to the measure $\pi$. In the following $T>0$ is fixed. 

\begin{defn} \label{def:conv_loi}
Let $(\pi_t^n)$ be a sequence of flows of  measures, each element belonging to the Skorokhod space $\mathcal{D}([0,T],\mathcal{M}_+)$. For each $n$, let $\mathcal{Q}_n$ be the probability measure on  $\mathcal{D}([0,T],\mathcal{M}_+)$ corresponding to $\{\pi_t^n\; ; \; t\in[0,T]\}$.
\begin{enumerate}
\item  We say that the  sequence $(\pi_t^n)$ \emph{converges to the deterministic flow} $\{\pi_t\}$ if the probability measure $\mathcal{Q}_n$ converges to
  the  Dirac probability measure concentrated on the deterministic flow $\{\pi_t\}$.
  \item Fix $t \in [0,T]$. We say that $(\pi_t^n)$ \emph{converges in probability} to the deterministic measure $\pi_t \in \mathcal{M}_+$ if, for all smooth test functions $G: \cX \to \R$, and all $\delta >0$,
\begin{equation}
   \mathcal{Q}_n\Big[ \big| \big\langle \pi_t^n, G \big\rangle -  \big\langle \pi_t, G \big\rangle \big| > \delta\Big]  \xrightarrow[n\to\infty]{} 0. \label{eq:conv}
\end{equation}
\end{enumerate}
\end{defn}

The next proposition gives the equivalence between the two notions above.

\begin{prop}
Let $(\pi_t^n)$ be a sequence of flows of measures which converges  to a deterministic flow $\{\pi_t\}$.
  Assume that 
$t \in [0,T] \mapsto \pi_t \in \mathcal{M}_+$ 
 is continuous (with respect to the weak topology). Then,
for any $t \in [0,T]$ fixed, $(\pi_t^n)$ converges in probability to $\pi_t \in \mathcal{M}_+$.  
\end{prop}

\begin{proof}
By assumption, the limiting probability measure on $\mathcal{D}([0,T],\mathcal{M}_+)$ is concentrated on weakly continuous trajectories. Therefore, the limiting flow is almost surely continuous, and the map 
$\{\pi_t\; ; \; t\in[0,T]\} \mapsto \pi_t$
is continuous from  $\mathcal{D}([0,T],\mathcal{M}_+)$ to $\mathcal{M}_+$. Then, for $t \in [0,T]$,  $(\pi_t^n)$ converges in distribution to $\pi_t$. Since the latter is deterministic, this induces  convergence in probability. 
\end{proof}

 We recall here the main result, that we are going to prove in a different way. Recall that $\pi_t^n$ is the empirical measure defined in \eqref{eq:pi} and  $\mathcal{Q}_n$ is the probability measure on  $\mathcal{D}([0,T],\mathcal{M}_+)$ corresponding to the flow $\{\pi_t^n\}$.
\begin{thm}[Hydrodynamic limits for the SSEP]\label{thm:ssep}
Fix a density profile $m_0: \T \to [0,1]$ and let $(\mu^n)$ be a sequence of probability measures such that, under $\mu^n$, the sequence $(\pi_0^n(d\theta))$ converges in probability to $m_0(\theta)d\theta$. In other words,
\[
\limsup_{n\to\infty}\mu^n \bigg[\Big|n^{-1} \sum_{i\in\T_n}G(i/n) \eta(i) - \int_\T G(\theta) m_0(\theta) d\theta \Big| > \delta\bigg]=0,
\]
for any $\delta >0$ and any smooth function $G: \T \to \R$. Assume moreover that this initial data is \emph{well-prepared}, in the sense that: 
\begin{equation}\lim_{n\to\infty} \frac{1}{n}\Ent_{\nu_{\alpha}^n}(\mu^n) = \int_{\T}{h(m_0(\theta))d\theta} - h\Big(\int_\T{m_0(\theta)d\theta}\Big),\label{eq:well-prep}\end{equation}
where $h$ has been defined in Example \ref{example}. Then, for any $t>0$, the sequence $\{\pi_t^n\}_{n \in \N}$ converges in probability to the deterministic measure $\pi_t(d\theta)=m(t,\theta)d\theta$ where $m$ is solution to the heat equation \eqref{eq:heat} on $\R_+\times\T$. The entropy also converges: 
$$\lim_{n\to\infty} \frac{1}{n}\Ent_{\nu^n_{\alpha}}(\mu^n_t) = \int_{\T}{h(m(t,\theta))d\theta} - h\Big(\int_\T{m(t,\theta)d\theta}\Big).$$

\end{thm}
Note that $\int m(t,\theta)d\theta$ is actually a constant, given by the fixed density of particles. The convergence of the entropy is equivalent to the local Gibbs behavior (see \cite{Kos}). Hence, the assumptions and conclusions are those obtained with the relative entropy method of \cite{yau}. However, the techniques and restrictions are the same as for the entropy method of \cite{GPV}: we do not use smoothness of solutions to the hydrodynamic PDE, but we use the replacement lemma (see Subsection \ref{sec:bounds}), which relies on the two-block estimate, rather than the one-block estimate alone as in \cite{yau}.

\subsection{The gradient flow approach to Theorem \ref{thm:ssep}}

We are going to apply Theorem \ref{theo:grad_flow} to obtain Theorem \ref{thm:ssep}. The main steps are as follows: 

\begin{enumerate}
\item We first need to prove that the sequence $(\mathcal{Q}_n)$ is relatively compact, so that there exists a converging subsequence. Such an argument was already part of the entropy method of \cite{GPV}. We refer to \cite{MR1707314}[Chapter 4, Section 2] for the proof in the context of the simple exclusion process.

\item In order to prove \eqref{eq:conv0}, we have to investigate the convergence of the relative entropy with respect to the invariant measure $\nu_\alpha^n$ towards the free energy associated to the limiting PDE \eqref{eq:heat}, which in our case reads as
\[
\mathcal{H}(m)=\int_\T{h(m(\theta))d\theta} - h\Big(\int_\T{m(\theta) d\theta}\Big).
\] This result is actually equivalent to the large deviation principle for  $\nu^n_{\alpha}$ (see for example \cite{Mar}), and is standard (see \cite{MR1707314}).
Moreover, if our initial data is close (in relative entropy) to a slowly varying Bernoulli product measure\footnote{This is also the assumption used  to make Yau's relative entropy method work, see \cite{yau}.} associated to $m$, which satisfies
$
\nu^n_{\rho(\cdot)}\{\eta(i)=1\}=m(i/n),
$ then its relative entropy with respect to $\nu_\alpha^n$ converges to the limiting free energy, so that we can easily have  \eqref{eq:well-prep}.

\item We prove the lower bound for the entropy production along curves \eqref{eq:bound_entropy}  and the lower bound for the slopes \eqref{eq:bound_slopes} in Subsection \ref{sec:bounds}. 

\item When passing to the limit, we obtain that for any weak limit $\mathcal{Q}$ of $(\mathcal{Q}_n)$,
\begin{multline*}
\Q\bigg[\int_\T h(m(T,\theta))d\theta - \int_\T h(m(0,\theta))d\theta + \\ \frac{1}{2}\int_0^T{\int_\T {m(1-m)\Big(\frac{\partial (h'(m))}{\partial \theta}\Big)^2d\theta}dt} + \frac{1}{2}\int_0^T{\Vert\dot m_t\Vert^2_{-1,m}\; dt}\bigg] \leq 0.\end{multline*}
Since the expression inside the expectation is the characterization of solutions to the heat equation as minimizing-movement curves, it is non-negative, and almost surely $m$ is a solution to the heat equation. Uniqueness of solutions starting from $m_0$ allows us to conclude.

\end{enumerate}



\subsection{Bounds and convergence}
\label{sec:bounds}

 Here we prove  that \eqref{eq:bound_entropy} and \eqref{eq:bound_slopes} are satisfied for the density $\rho_{t}^n$ of the SSEP accelerated in time, assuming that the empirical measure $(\pi_t^n(d\theta))$ converges  to a deterministic curve $m_t(\theta)d\theta$.
Let us start with \eqref{eq:bound_slopes}. The argument is based on a duality argument (Proposition \ref{prop:ineq}) and on the replacement lemma (Lemma \ref{lem:replacement}) which is commonly used in the literature (see for example \cite{MR1707314}).

\begin{prop}\label{prop:ineq} Consider a couple $(\rho_t, \psi_t)$ satisfying the continuity equation \eqref{continuity_equation} for almost every $t\ge 0$. For any smooth (in time) function $J: [0,T] \times \mathcal{X} \to \R$, 
\begin{multline*} \int_0^T{\mathcal{A}(\rho_t, \psi_t)dt} \geq 2\sum_{x\in\cX} J(T,x)\rho_T(x)\nu(x) - 2\sum_{x\in\cX} J(0,x)\rho_0(x)\nu(x)  \\
- 2\int_0^T{\sum_{x\in\cX} \partial_t{J}(t,x)\rho_t(x) \nu(x)dt} -\int_0^T{\sum_{x,y} (J(t,x)-J(t,y))^2\hat{\rho}_t(x,y)K(x,y)\nu(x)dt}. 
\end{multline*}

\end{prop}

\begin{proof}
From the continuity equation \eqref{continuity_equation}, we have
\begin{align*}
& \sum_{x\in\cX} J(T,x)\rho_T(x)\nu(x) - \sum_{x\in\cX} J(0,x)\rho_0(x)\nu(x)  \\ & = \int_0^T{\sum_{x\in\cX} \partial_t{J}(t,x)\rho_t(x) \nu(x) + J(t,x)\dot{\rho}_t(x)\nu(x) dt}   = \int_0^T{\sum_{x\in\cX} \partial_t{J}(t,x)\rho_t(x) \nu(x)dt}\\
& \qquad \qquad - \int_0^T{\sum_{x,y\in\cX} J(t,x)(\psi_t(y) - \psi_t(x))\hat{\rho}_t(x,y)K(x,y)\nu(x)dt}.
\end{align*}
We symmetrize in $x$ and $y$ the last term, and get
\begin{multline*}
2\sum_{x\in\cX} J(T,x)\rho_T(x)\nu(x) - 2\sum_{x\in\cX} J(0,x)\rho_0(x)\nu(x) = 2\int_0^T{\sum_{x\in\cX} \partial_t{J}(t,x)\rho_t(x) \nu(x)dt}  \\
- \int_0^T{\sum_{x,y\in\cX} (J(t,x)-J(t,y))(\psi_t(y) - \psi_t(x))\hat{\rho}_t(x,y)K(x,y)\nu(x)dt}, 
\end{multline*}
and therefore 
\begin{multline*}
\frac{1}{2n}\int_0^T{\mathcal{A}(\rho_t, \psi_t)dt} 
\geq \frac{1}{n}\sum_{x\in\mathcal{X}} J(T,x)\rho_T(x)\nu(x) - \frac{1}{n}\sum_{x\in\mathcal{X}} J(0,x)\rho_0(x)\nu(x)\\
  - \frac{1}{n}\int_0^T{\sum_{x\in\mathcal{X}} \partial_t{J}(t,x)\rho_t(x) \nu(x)dt} 
 -\frac{1}{2n}\int_0^T{\sum_{x,y} (J(t,x)-J(t,y))^2\hat{\rho}_t(x,y)K(x,y)\nu(x)dt.} 
\end{multline*}\end{proof}

To apply Proposition \ref{prop:ineq} to the SSEP, we consider  observables of the form
\begin{equation} J(t,\eta) = \sum_{i\in\T_n} G\Big(t,\frac i n\Big)\eta(i)\label{eq:J}\end{equation} for smooth functions $G:[0,T] \times \T\to\R$. For any $\ell \in \N$ and  $i\in \T_n$, we denote by $\eta^\ell(i)$ the empirical density of particles in a box of size $2\ell+1$ centered at $i$: \[ \eta^\ell(i):=\frac{1}{2\ell+1} \sum_{|j-i|\le \ell} \eta(i).\]
Hereafter we also denote by  $\tau_x$ the translated operator that acts on local functions $g:\{0,1\}^{\mathbb{Z}}\to\R$ as $(\tau_x g)(\eta):=g(\tau_x \eta)$, and $\tau_x \eta$ is the configuration obtained from $\eta$ by shifting: $(\tau_x \eta)_y=\eta_{x+y}$. The main tool that we are going to use is the well-known {replacement lemma}, which is a consequence of the averaging properties of the SSEP. We recall the main statement and refer the reader to \cite{GPV,MR1707314} for a proof:

\begin{lem}[Replacement Lemma]\label{lem:replacement}
Denote by $\mathbb{P}_{\mu^n}$ the probability measure on the Skorokhod space $\mathcal{D}([0,T],\cX_n)$ induced by the Markov process $\{\eta_{t}^n\}_{t\geq 0}$ starting from $\mu^n$. Then, for every $\delta >0$ and every local function $g$,
\[\limsup_{\varepsilon \to 0} \limsup_{n\to\infty} \mathbb{P}_{\mu^n}\bigg[\int_0^T n^{-1} \sum_{x\in\T_n} \tau_x V_{\varepsilon n}(\eta_{s})ds \ge \delta \bigg] = 0, \]
where 
\[ V_\ell(\eta)=\bigg|\frac{1}{2\ell+1} \sum_{|y|\le \ell} \tau_yg(\eta) - \tilde g (\eta^\ell(0))\bigg| 
\]
and $\tilde g : (0,1) \to \R$ corresponds to the expected value: $\tilde g(\alpha):= \int g(\eta) d\nu_\alpha(\eta).$
\end{lem}

We are now able to conclude the proof. We treat separately the terms in the right-hand side of Proposition \ref{prop:ineq}, taking $J$ as in \eqref{eq:J}. Since, for any fixed $t$,  $(\pi_t^n(d\theta))$ converges in probability to $\pi_t(d\theta)=m_t(\theta)d\theta$ we have
\[
\frac{1}{n}\sum_{\eta \in \mathcal{X}_n} J(T,\eta)\rho_{T}^n(\eta)\nu_\alpha^n(\eta) = \cQ_n\big[\langle \pi_T^n, G\rangle\big]  \xrightarrow[n\to\infty]{} \int_\T G(T,\theta) m_T(\theta) d\theta.
\]
And the same happens at initial time for $\rho_0^n$. Similarly,
$$\frac{1}{n}\int_0^T \sum_{\eta \in \cX_n}  \frac{\partial J}{\partial t}(t,\eta) \rho^n_t(\eta)\nu_\alpha^n(\eta)dt \longrightarrow \int_0^T{\int_{\T}{\frac{\partial G}{\partial t}(t,\theta)m_t(\theta)d\theta}dt}.$$
Then, we write 
\begin{multline*}
\int_0^T \frac{1}{n}\sum_{\eta,\eta'\in\cX_n} (J(t,\eta)-J(t,\eta'))^2\; \widehat{\rho_{t}^n}(\eta,\eta')K_n(\eta,\eta')\nu_\alpha^n(\eta)dt \\
 = \frac{n^2}{n}\int_0^T \sum_{\eta, i} \Big[G\Big(t,\frac{i}{n}\Big)-G\Big(t,\frac{i+1}{n}\Big)\Big]^2 \; {\eta(i)(1-\eta(i+1))}\; \widehat{\rho_{t}^n}(\eta,\eta^{i,i+1}) \nu_\alpha^n(\eta) dt.
\end{multline*}
We now use the logarithmic inequality \eqref{eq:logarith_ineq} and write that the latter is smaller than
\[
n\int_0^T \sum_{\eta,i} \Big[G\Big(t,\frac{i}{n}\Big)-G\Big(t,\frac{i+1}{n}\Big)\Big]^2 \;  {\eta(i)(1-\eta(i+1))}\; \frac{\rho_{t}^n(\eta)+\rho_{t}^n(\eta^{i,i+1})}{2} \nu_\alpha^n(\eta) dt.
\]
From the invariance property of $\nu_\alpha^n$ with respect to the change of variables $\eta \to \eta^{i,i+1}$, and from the smoothness of $G$ we get that the above quantity is equal to
\begin{equation}\frac{1}{n}\int_0^T \sum_{\eta \in \cX_n} \sum_{i \in \T_n} \Big[G'\Big(t,\frac{i}{n}\Big)\Big]^2 \;  {\eta(i)(1-\eta(i+1))}\; \rho_{t}^n(\eta) \nu_\alpha^n(\eta) dt + o\Big(\frac{1}{n}\Big), \label{eq:rep1}\end{equation}
where $G'$ denotes the space derivative of $G$. 

Above we want to replace $\eta(i)(1-\eta(i+1))$ by $m(i/n)(1-m(i/n))$. For $\varepsilon >0$ we define the approximation of the identity $i_\varepsilon(u)=(2\varepsilon)^{-1} \mathbf{1}\{|u|\le \varepsilon\}.$
With that notation, $\eta_{t}^{\varepsilon n}(0)$ is very close to $\langle \pi_t^n, i_\varepsilon \rangle$. Let us denote $h(\eta):=\eta(0)(1-\eta(1))$. Since $G$ is a smooth function, \eqref{eq:rep1} equals 
\[
\frac{1}{n}\int_0^T \sum_{\eta \in \cX_n} \sum_{i \in \T_n} \frac{1}{2\varepsilon n+1}\sum_{|j-i|\le \varepsilon n}\Big[G'\Big(t,\frac{j}{n}\Big)\Big]^2 \; \tau_ih(\eta) \rho_{t}^n(\eta) \nu_\alpha^n(\eta) dt + O(\varepsilon^2).
\] 
A summation by parts shows that the previous term can be written as 
\[
\frac{1}{n}\int_0^T \sum_{\eta \in \cX_n} \sum_{i \in \T_n} \Big[G'\Big(t,\frac{i}{n}\Big)\Big]^2  \frac{1}{2\varepsilon n+1}\sum_{|j-i|\le \varepsilon n} \tau_jh(\eta) \rho_{t}^n(\eta) \nu_\alpha^n(\eta) dt + O(\varepsilon^2).
\]
By Lemma \ref{lem:replacement}, this expression is then equal to 
\[
\frac{1}{n}\int_0^T \sum_{\eta \in \cX_n} \sum_{i \in \T_n} \Big[G'\Big(t,\frac{i}{n}\Big)\Big]^2   \tau_i\tilde h\big(\langle \pi_t^n(\eta),i_\varepsilon\rangle\big)  \nu_\alpha^n(\eta) dt + R_{n,\varepsilon,T},
\] 
 where $R_{n,\varepsilon,T}$ vanishes in probability as $n$ goes to infinity and then $\varepsilon$ goes to 0. From the convergence in probability of $(\pi_t^n)$, the last expression converges to
\[
\int_0^T \int_\T m_t(\theta)(1-m_t(\theta)) (G'(t,\theta))^2 \; d\theta dt.
\]
As a result, since the convergences above are valid for any smooth function $G$, 
\begin{align*}
\liminf_{n \to \infty}  \frac{1}{n}\int_0^T{\mathcal{A}(\rho_t^n, \psi_{t}^n)dt} & \geq \int_0^T  \sup_G\bigg\{ 2\int_\T G\;\dot m_t\;d\theta - \int_\T m_t(1-m_t)(G')^2d\theta\bigg\}dt\\
& =\int_0^T{\Vert\dot m_t\; \Vert^2_{-1,m}\; dt}.\end{align*}
In the same way, we need to prove
$$\liminf_{n\to\infty} \frac{1}{n}\int_0^T{\mathcal{E}(\rho_t^n, \log \rho_t^n)dt} \geq \int_0^T{\int_\T {m(1-m)\Big(\frac{\partial (h'(m))}{\partial \theta}\Big)^2d\theta}dt}.$$
Since the arguments are essentially the same as for the slopes, we shall be more brief in the exposition. We denote
\[
\nabla_nG\Big(t,\frac{i}{n}\Big)=n\Big[G\Big(t,\frac{i+1}{n}\Big) - G\Big(t,\frac{i}{n}\Big)\Big].
\] By duality, we have
\begin{align*}
&\frac{1}{2n}\int_0^T{\mathcal{E}(\rho_t^n, \log \rho_t^n)dt}  \\
& \geq \int_0^T \sum_{\eta, i}(\log \rho^n_t(\eta^{i,i+1}) - \log \rho^n_t(\eta))\nabla_n G\Big(t,\frac{i}{n}\Big)\eta_i(1-\eta_{i+1})\widehat{\rho^n_t}(\eta, \eta^{i,i+1})\nu_{\alpha}^n(\eta)dt \\
&\quad -\frac1{2n}\int_0^T{\sum_{\eta, i}\Big[\nabla_nG\Big(t,\frac{i}{n}\Big)\Big]^2\eta_i(1-\eta_{i+1})\widehat{\rho^n_t}(\eta, \eta^{i,i+1})\nu_{\alpha}^n(\eta)dt} \\
&\geq  \int_0^T{\sum_{\eta, i}(\rho^n_t(\eta^{i,i+1}) - \rho^n_t(\eta))\nabla_nG\Big(t,\frac{i}{n}\Big)\eta_i(1-\eta_{i+1})\nu_{\alpha}^n(\eta)dt} \\
&\quad -\frac{1}{2n}\int_0^T{\sum_{\eta, i} \Big[\nabla_nG\Big(t,\frac{i}{n}\Big)\Big]^2\eta_i(1-\eta_{i+1})\rho_t^n(\eta)\nu_{\alpha}^n(\eta)dt} \\
&= -\int_0^T{\sum_{\eta, i}\rho_t^n(\eta)\Big[\nabla_nG\Big(t,\frac{i}{n}\Big)-\nabla_nG\Big(t,\frac{i-1}{n}\Big)\Big]\eta_i\nu_{\alpha}^n(\eta)dt} \\
&\quad -\frac{1}{2n}\int_0^T{\sum_{\eta, i} \Big[\nabla_nG\Big(t,\frac{i}{n}\Big)\Big]^2\eta_i(1-\eta_{i+1})\rho_t^n(\eta)\nu_{\alpha}^n(\eta)dt}
\end{align*}
Using the replacement lemma, passing to the supremum in $G$, and to  the limit,
\begin{align*}
&\liminf_{n\to\infty}\frac{1}{n}\int_0^T{\mathcal{E}(\rho_t^n, \log \rho_t^n)dt} \\
&\geq \underset{G}{\sup}\bigg\{ -2\int_0^T{\int_\T{G''m_t \; d\theta}dt} - \int_0^T{\int_\T{m_t(1-m_t)(G')^2d\theta}dt}\bigg\} \\
&= \underset{G}{\sup} \bigg\{-2\int_0^T{\int_\T{G''m_t(1-m_t)h''(m_t)d\theta}dt} - \int_0^T{\int_\T{m_t(1-m_t)(G')^2d\theta}dt}\bigg\}\\
&=\int_0^T{\int_\T{m_t(1-m_t)\Big(\frac{\partial h'(m_t)}{\partial \theta}\Big)^2d\theta}dt},
\end{align*}
and this is exactly what we were seeking to prove.


\begin{acknowledgement}
M.F. : I would like to thank Hong Duong, Matthias Erbar, Vaios Laschos and Andr\'e Schlichting for discussions on convergence of gradient flows. Part of this work was done while I was staying at the Hausdorff Institute for Mathematics in Bonn, whose support is gratefully acknowledged. I also benefited from funding from GDR MOMAS and from NSF FRG grant DMS-1361185.

M.S.: This work has been supported by the French Ministry of Education through the grant ANR (EDNHS), and also  by  CAPES (Brazil) and IMPA (Instituto de Matematica Pura e Aplicada, Rio de Janeiro) through a post-doctoral fellowship. 
\end{acknowledgement}

\end{document}